\documentclass[a4paper,11pt]{amsart}

\usepackage{amsmath}
\usepackage{amsthm}
\usepackage{amssymb}
\usepackage{stmaryrd}
\usepackage{MnSymbol}
\usepackage{bbm}
\usepackage{bm}
\usepackage[top=30truemm, bottom=30truemm, left=25truemm, right=25truemm]{geometry}

\numberwithin{equation}{section}

\newtheorem{theorem}{Theorem}[section]
\newtheorem{lemma}[theorem]{Lemma}

\newtheorem{proposition}[theorem]{Proposition}
\newtheorem{conjecture}[theorem]{Conjecture}

\theoremstyle{definition}
\newtheorem{remark}[theorem]{Remark}
\newtheorem*{acknowledgment}{Acknowledgment}

\renewcommand{\Re}{\mathrm{Re}\hspace{1pt}}
\renewcommand{\Im}{\mathrm{Im}\hspace{1pt}}

\newcommand{\meas}{\mathrm{meas}} 
\newcommand{\m}{\mathbf{m}}
\newcommand{\supp}{\mathrm{supp}}
\newcommand{\dist}{\mathrm{dist}}

\title[Universality for certain Dirichlet series]{\large Universality for the iterated integrals of logarithms of \textit{L}-functions in the Selberg class}
\author[K. Nakai]{\large Keita Nakai}
\date{}

\begin{document}

\begin{abstract}
We prove the universality theorem for the iterated integrals of logarithms of $L$-functions in the Selberg class on some line parallel to the real axis.
\end{abstract}

\maketitle

\section{Introduction and statement of main results}

Let $s = \sigma + it$ be a complex variable. 
The Riemann zeta-function $\zeta(s)$ is defined by the infinite series 
$\sum_{n=1}^{\infty} n^{-s}$ 
in the $\sigma > 1$, 
and can be continued meromorphically to the whole plane $\mathbb{C}$. 
Bohr and Courant \cite{BC} proved that the set $\{\zeta(\sigma + it): t \in \mathbb{R}\}$ is dense in $\mathbb{C}$ for  $1/2 < \sigma < 1$, and Bohr \cite{Bo} proved that the set $\{\log{\zeta}(\sigma + it): t \in \mathbb{R}\}$ is dense in $\mathbb{C}$ for  $1/2 < \sigma < 1$. 
The problem that the set $\{\zeta(1/2 + it): t \in \mathbb{R}\}$ is dense or not in $\mathbb{C}$ is still open. 
For this reason, Endo and Inoue \cite{EI} considered the following functions $\tilde{\eta}_{m}(s)$ defined by 
\[
\tilde{\eta}_m(\sigma +it) = \int_{\sigma}^{\infty} \tilde{\eta}_{m-1}(\alpha + it)\, d\alpha,\ 
\tilde{\eta}_0(\sigma +it) = \log{\zeta(\sigma + it)} = \int_{\infty}^{\sigma} \frac{\zeta'}{\zeta}(\alpha + it)\, d\alpha
\]
for $\sigma$ which is not equal to the real part of the zeros of $\zeta(s)$ and  proved the following density result.
\begin{theorem}[Endo--Inoue \cite{EI}]
Let $1/2 \le \sigma < 1$, and m be a positive integer. Then the set 
\[
\{\tilde{\eta}_{m}(\sigma + it) : t \in \mathbb{R}\}
\]
is dense in $\mathbb{C}$.
\end{theorem}  
 
On the other hand, in 1975, Voronin proved the following theorem which is called the universality theorem. 

\begin{theorem} [Voronin \cite{Vo}]

Let $\mathcal{K}$ be a compact set in the strip $1/2 < \sigma < 1$ with connected complement, and let $f(s)$ be a non-vanishing continuous function on $\mathcal{K}$ that is analytic in the interior of $\mathcal{K}$. Then, for any $\varepsilon > 0$ 
\[
\liminf_{T \to \infty} \frac{1}{T} \meas \left\{\tau \in [0, T] :  \sup_{s \in \mathcal{K}} |\zeta(s + i\tau) -f(s)| < \varepsilon \right\} > 0,
\]
where $\meas$ denotes the Lebesgue measure.
\end{theorem}

This universality theorem has been improved and extended in various zeta-functions and $L$-functions. 
For example, Laurin\v{c}ikas and Matsumoto \cite{LM} proved the universality theorem for autmorphic $L$-function attached to cusp forms for $SL(2, \mathbb{Z})$ using the new method which is called the positive density method. 
As another extension, Endo \cite{En} proved the universality theorem for $\tilde{\eta}_m(s)$.
In this paper, we extend Endo's result to the Selberg class.

The Selberg class $\mathcal{S}$ is defined by Selberg \cite{Se} which is the class of Dirichlet series $\mathcal{L}(s) = \sum_{n=1}^{\infty} a(n) n^{-s} $ satisfying the following assumptions.

\begin{itemize} 
 \item[(S1)] Ramanujan hypothesis: $a(n) \ll_{\varepsilon} n^{\varepsilon}$ for any $\varepsilon > 0$.
 \item[(S2)] Analytic continuation: There exists a nonnegative integer $k$ such that $(s-1)^k \mathcal{L}(s)$ is an entire function of finite order.
 \item[(S3)] Functional equation: $\mathcal{L}(s)$ satisfies a functional equation of type 
 \[
 \Phi_{\mathcal{L}}(s) = \omega \overline{\Phi_{\mathcal{L}}(1 - \overline{s})},
 \]
 where 
 \begin{equation*}
 \Phi_{\mathcal{L}}(s) = \mathcal{L}(s)Q^s \prod_{j=1}^{f}\Gamma(\lambda_j s + \mu_j),
 \end{equation*}
 with positive real numbers $Q$, $\lambda_j$ and 
 complex numbers $\mu_j$ and $\omega$ with $\Re(\mu_j) \ge 0$ and $|\omega| =1$.
 \item[(S4)] Euler product: $\mathcal{L}(s)$ has a product representation 
 \[
 \mathcal{L}(s) = \prod_{p} \mathcal{L}_{p}(s),
 \] 
 where 
 \[
 \log{\mathcal{L}_{p}}(s) = \sum_{k=1}^{\infty} \frac{b(p^k)}{p^{ks}}
 \]
 with suitable coefficients $b(p^k)$ satisfying $b(p^k) \ll p^{k\theta}$ for some $0 \le \theta < 1/2$. 
\end{itemize}

The degree of $\mathcal{L} \in \mathcal{S}$ is defined by 
\[
d_{\mathcal{L}} = 2 \sum_{j=1}^{f} \lambda_j.
\]
This is unique although the functional equation is not since the analogue of the Riemann--von Mangoldt formula of $\mathcal{L}$ holds and $d_\mathcal{L}$ is contained in the coefficient of the main term of this formula (cf. \cite[Theorem~7.7]{St}). 
The following universality theorem for $L$-functions in the Selberg class can be proved by applying the positive density method.

\begin{theorem} [Nagoshi--Steuding \cite{NS}] \label{NS uni}
Let $\mathcal{L} \in \mathcal{S}$ satisfying
\[
\lim_{x \to \infty} \frac{1}{\pi(x)} \sum_{p \le x} |a(p)|^2 = \kappa,
\]
where $\kappa$ is some positive constant(may depend on $\mathcal{L}$). 
Let $\mathcal{K}$ be a compact set in the strip $\max\left\{\frac{1}{2}, 1 - \frac{1}{d_\mathcal{L}} \right\} < \sigma < 1$ with connected complement, and let $f(s)$ be a non-vanishing continuous function on $\mathcal{K}$ that is analytic in the interior of $\mathcal{K}$. Then, for any $\varepsilon > 0$ 
\[
\liminf_{T \to \infty} \frac{1}{T} \meas \left\{\tau \in [0, T] :  \sup_{s \in \mathcal{K}} |\mathcal{L}(s + i\tau) -f(s)| < \varepsilon \right\} > 0.
\]
\end{theorem}

Therorem~\ref{NS uni} shows the universality for $L$-function in the Selberg class only for $\max\left\{\frac{1}{2}, 1 - \frac{1}{d_\mathcal{L}} \right\} < \sigma < 1$ because we require mean value theorem for $L$-function in the Selberg class such that
\[
\int_{0}^{T} |\mathcal{L}(\sigma + it)|^2\, dt \ll T,\ T \to \infty
\]
to prove this theorem but we know this formula only for $\max\left\{\frac{1}{2}, 1 - \frac{1}{d_\mathcal{L}} \right\} < \sigma < 1$.

We see that $\mathcal{L}(s) = \sum_{n=1}^{\infty} a(n) n^{-s}$ is absolutely convergence for $\sigma > 1$ from (S1), and for $\sigma > 1$, 
\begin{equation} \label{Dirichlet rep}
-\frac{\mathcal{L}'}{\mathcal{L}}(s) = \sum_{n=1}^{\infty} \frac{\Lambda_{\mathcal{L}}(n)}{n^{s}},
\end{equation}
where 
\begin{align*}
\Lambda_{\mathcal{L}}(n) &= 
\begin{cases}
b(p^k)\log p^k & (n = p^k, \text{where $p$ is a prime and $k$ is a positive integer}), \\
0 & (\text{otherwise})
\end{cases}
\end{align*}
which can be shown by (S4) and logarithmic differentiation of $\mathcal{L}$

We define the iterated integrals of the logarithm of $L$-function in the Selberg class $\tilde{H}_m(s)$ as 
\[
\tilde{H}_{m}(\sigma + it) = \int_{\sigma}^{\infty} \tilde{H}_{m-1}(\alpha + it)\, d\alpha,
\]
\[
\tilde{H}_{0}(\sigma + it) =\log{\mathcal{L}(\sigma + it)} =  \int_{\infty}^{\sigma} \frac{\mathcal{L}'}{\mathcal{L}}(\alpha + it)\, d\alpha
\]
for $m \ge 1$ and $s \in G_\mathcal{L}$. 
Here $G_\mathcal{L}$ is defined by 
\[
G_{\mathcal{L}} = \mathbb{C} \setminus \left\{ \left(\bigcup_{\rho_{\mathcal{L}} = \beta_{\mathcal{L}} + i \gamma_{\mathcal{L}}} \{s = \sigma + i \gamma_{\mathcal{L}} \ |\ \sigma \le \beta_\mathcal{L} \}\right) \cup I_\mathcal{L} \right\}, 
\]
\begin{align*}
I_{\mathcal{L}} &= 
\begin{cases}
(-\infty, 1] & (\mbox{$\mathcal{L}$ has a pole at $s=1$}), \\
\emptyset & (\mbox{$\mathcal{L}$ is an entire}),
\end{cases}
\end{align*}
where $\rho_\mathcal{L} = \beta_\mathcal{L} + i\gamma_\mathcal{L}$ runs over zeros of $\mathcal{L}(s)$.
It follows from (\ref{Dirichlet rep}) that  
\[
\tilde{H}_{m}(s) = \sum_{n=2}^{\infty} \frac{\Lambda_{\mathcal{L}}(n)}{(\log n)^{m+1}n^s}
\]
for $\sigma > 1$, and this is absolutely convergent for $\sigma > 1$. 

Let $\mathcal{S}'$ be the subclass, which is characterized by the following two additional assumptions for the Selberg class. 

\begin{itemize}
 \item[(S5)] Zero density estimate: Let $T > 0$, and $N_\mathcal{L}(\sigma, T)$ denotes the number of zeros $\rho_\mathcal{L} = \beta_\mathcal{L} + i\gamma_\mathcal{L}$ of $\mathcal{L}$ satisfying $\beta_\mathcal{L} > \sigma$ and $|\gamma_\mathcal{L}| < T$.
 Then, there exists $\sigma_\mathcal{L} \ge 1/2$ such that
 \begin{equation}
 N_\mathcal{L}(\sigma, T) = o_\mathcal{L} \left(\frac{T}{(\log{T})^\alpha}\right)
 \label{zero density}
 \end{equation}
 uniformily for $\sigma_\mathcal{L} < \sigma < 1$ for any $\alpha > 0$
 \item[(S6)] Prime mean-square: There exists a constant $\kappa = \kappa(\mathcal{L}) > 0$ such that 
 \begin{equation}
\lim_{x \to \infty} \frac{1}{\pi(x)} \sum_{p \le x} |a(p)|^2 = \kappa.
\end{equation}
\end{itemize} 
We expect that the following conjecture is true. 

\begin{conjecture} [Grand Density Hypothesis]
There is some positive constant $c_\mathcal{L}$ such that for $\sigma > 1/2$
\begin{equation} \label{GDH}
N_\mathcal{L}(\sigma, T) \ll T^{1 - c_\mathcal{L}(\sigma - 1/2) + \varepsilon}.
\end{equation}
\end{conjecture}
Then, assuming the Grand density hypothesis, we can take $\sigma_\mathcal{L} = 1/2$. 

When $\mathcal{L} \in \mathcal{S}'$, we define the strip $\mathcal{D}_\mathcal{L} = \{s = \sigma + it : \sigma_\mathcal{L} < \sigma < 1\}$. 
Now we state the main theorem of this paper. 

\begin{theorem} \label{main theorem} 
Let $\mathcal{L} \in \mathcal{S}'$ and $m$ be a nonnegative integer.
Let $\mathcal{K}$ be a compact set in $\mathcal{D}_\mathcal{L}$ with connected complement, and let $g(s)$ be a continuous function on $\mathcal{K}$ that is analytic in the interior of $\mathcal{K}$. Then, for any $\varepsilon > 0$, 
\[
\liminf_{T \to \infty} \frac{1}{T} \meas \left\{\tau \in [T, 2T] :  \sup_{s \in \mathcal{K}} |\tilde{H}_{m}(s +i\tau) - g(s)| < \varepsilon \right\} >0.
\]
\end{theorem}

Kaczorowski and Perelli \cite{KP} proved 
\[
 N_\mathcal{L}(\sigma, T) \ll T^{4(d_\mathcal{L} + 3)(1 - \sigma) +\varepsilon}
\]
uniformly for $1/2 \le \sigma < 1$ unconditionally, so we can put 
\[
\sigma_\mathcal{L} = 1 - \frac{1}{4(d_\mathcal{L} + 3)}.
\]
Therefore, we can prove the universality theorem for 
$\tilde{H}_m(s)$ for $1 - \frac{1}{4(d_\mathcal{L} + 3)} < \sigma < 1$ under only (S6). 

\begin{remark} 
We can take $\sigma_\zeta = 1/2$ (for example see \cite[Theorem~9.19 A]{Ti}). 
Therefore Theorem~\ref{main theorem} contains Endo's result.
\end{remark}

Using Theorem~\ref{main theorem} as $m = 0$ under the Grand density hypothesis, we can prove the universality for $L$-function in the Selberg class for $1/2 < \sigma < 1$.
 
\begin{theorem} \label{main2}
Let $\mathcal{L} \in \mathcal{S}$ satisfying the Grand density hypothesis and $(S6)$.
Let $\mathcal{K}$ be a compact set in the strip $1/2 < \sigma < 1$ with connected complement, and let $f(s)$ be a non-vanishing continuous function on $\mathcal{K}$ that is analytic in the interior of $\mathcal{K}$. Then, for any $\varepsilon > 0$ 
\[
\liminf_{T \to \infty} \frac{1}{T} \meas \left\{\tau \in [T, 2T] :  \sup_{s \in \mathcal{K}} |\mathcal{L}(s + i\tau) - f(s)| < \varepsilon \right\} > 0.
\]
\end{theorem}

\begin{remark}
Mishou \cite{Mi} proved the universality for Hecke $L$-function of algebraic number field $K$ 
in the strip $1 - d^{-1} < \sigma < 1$, where $K \ne \mathbb{Q}$, and $d = [K : \mathbb{Q}]$. 
Lee \cite{Le} proved the same universality result for $1/2 < \sigma < 1$ under the zero density estimate similar to (\ref{GDH}). 
Hecke $L$-function is in the Selberg class, so Theorem~\ref{main2} is analogue of Lee's result. 
However, own proof is quite different from Lee's proof.
\end{remark}

\section{PRELIMINARIES}

Let $\mathcal{L} \in \mathcal{S}'$, and we fix compact subset $\mathcal{K}$ satisfying the assumptions of Theorem~\ref{main theorem}. 
We introduce the following notations.
\begin{itemize}
 \item $|\mathcal{K}| = \max_{s \in \mathcal{K}} \Im(s) - \min_{s \in \mathcal{K}} \Im(s)$.
 \item $\tau_0 = \tau_0(\mathcal{K}) = (\max_{s \in \mathcal{K}}\Im(s) + \min_{s \in \mathcal{K}} \Im(s))/2$. 
 \item $\sigma_0 =\sigma_0(\mathcal{K}) = (\sigma_\mathcal{L} + \min_{s \in \mathcal{K}} \Re(s))/2$.
 \item For any $\Delta >0$, let
 \begin{equation}
 \mathcal{G}_{\sigma_0, \Delta} = \mathbb{R} \setminus \left\{ \left( \bigcup_{\substack{\rho_\mathcal{L} = 
  \beta_\mathcal{L} + i\gamma_\mathcal{L} \\ \beta_\mathcal{L} > \sigma_0}} (\gamma_\mathcal{L} 
  -\tau_0(\mathcal{K})- \Delta, \gamma_\mathcal{L} - \tau_0(\mathcal{K}) + \Delta) \right) 
  \cup   (-\tau_0(\mathcal{K}) - \Delta, -\tau_0(\mathcal{K}) + \Delta ) \right\}.
 \end{equation}

 \item For any $T >0$, let $\mathcal{I}_\mathcal{K}(T) = \mathcal{G}_{\sigma_0, |\mathcal{K}| +1} \cap [T, 2T]$. 
\end{itemize}

By the zero density estimate (S5), we have $\meas(\mathcal{I}_\mathcal{K}(T)) \sim T$ as $T \to \infty$. 
If $\tau \in \mathcal{G}_{\sigma_0, |\mathcal{K}| + 1}$, then $\mathcal{K} +i\tau$ is in $G_\mathcal{L}$.
Therefore, for $\tau \in \mathcal{G}_{\sigma_0, |\mathcal{K}| + 1}$, 
$\tilde{H}_m(s + i\tau)$ is holomorphic in $s \in \mathcal{K}$. 

We fix $\sigma_1,\ \sigma_2$ such that  
\[
\sigma_\mathcal{L} < \sigma_0 < \sigma_{1} < \min_{s \in \mathcal{K}} \Re(s),\ \max_{s \in \mathcal{K}} \Re(s) < \sigma_{2} < 1.
\]

Then, for any compact subset $\mathcal{K}$, we define the rectangle region $\mathcal{R}$ by 
\begin{equation} \label{definition of R}
\mathcal{R} = (\sigma_1,\ \sigma_2) \times i \left( \min_{s \in \mathcal{K}} \Im(s) - 1/2,\ \max_{s \in \mathcal{K}}\Im(s) + 1/2 \right).
\end{equation}
Let $\mathcal{H}(\mathcal{R})$ be the set of all holomorphic functions on $\mathcal{R}$. 
Then there exists a sequence of compact subsets $K_j$ of $\mathcal{R}$, $j=1, 2, \dots$ with the properties: 
\begin{itemize} 
 \item $\mathcal{R} = \bigcup_{j=1}^{\infty} K_j$, 
 \item $K_j \subset K_{j+1}$ for any $j \in \mathbb{N}$,
 \item For all compact subset $K$ of $\mathcal{R}$, there exists $j \in \mathbb{N}$ such that $K \subset K_j$ 
\end{itemize}
(See \cite[Chapter VII, 1.2 Proposition]{Co}).
Now, for $f, g \in \mathcal{H}(\mathcal{R})$, let $d_j(f, g) = \sup_{s \in K_j} |f(s) - g(s) |$ and put  
\[
d(f, g) = \sum_{j=1}^{\infty} \frac{1}{2^j} \frac{d_j(f, g)}{1 + d_j(f, g)}.
\]
Then $\mathcal{H}(\mathcal{R})$ is a metric space where $d$ is a metric on $\mathcal{H}(\mathcal{R})$,
and in particular, the metric space $\mathcal{H}(\mathcal{R})$ is a topological space with a topology of compact convergence. 
We put 
$\gamma = \{ s \in \mathbb{C} : |s| = 1 \}$ and write $\mathcal{B}(T)$ for the Borel set of $T$ which is a topological space.  

For any prime $p$, we put $\gamma_p = \gamma$ and  $\Omega = \prod_p \gamma_p.$
Since $\Omega$ is a compact topological abelian group, there exists the probability Haar measure $\m$ on $(\Omega, \mathcal{B}(\Omega))$. 
Then $\m$ is written by $\m = \otimes_p \m_p$, where $\m_p$ is the probability Haar measure on $(\gamma_p, \mathcal{B}(\gamma_p))$.
 
Let $\omega(p)$ be the projection of $\omega \in \Omega$ to the coordinate space $\gamma_p$.
$\{\omega(p) : \text{$p$ prime}\}$ is a sequence of independent complex-valued random elements defined on the probability space $(\Omega, \mathcal{B}(\Omega), \m)$. 
For $\omega \in \Omega$, we put $\omega(1) :=1$, 
\[
\omega(n) := \prod_p \omega(p)^{\nu(n ; p)}, 
\]
where $\nu(n ; p)$ is the exponent of the prime $p$ in the prime factorization of $n$.
Here, we define the $\mathcal{H}(\mathcal{R})$-valued random elements
\begin{equation}
\begin{split}
\tilde{H}_m(s, \omega) &:= \sum_{n=2}^{\infty} \frac{\Lambda_{\mathcal{L}}(n) \omega(n)}{(\log n)^{m+1}n^s} \\
                     &= \sum_p \sum_{k=1}^{\infty} \frac{b(p^k) \omega(p)^k}{(k\log{p})^m p^{ks}}.
\label{definition H}
\end{split}
\end{equation}
We remark that the second equation holds in the sense of analytic continuation for $\sigma > 1/2$. 

We define the probability measures on $(\mathcal{H}(\mathcal{R}), \mathcal{B}(\mathcal{H}(\mathcal{R})))$ by 
\begin{equation*}
\mathcal{Q}_T(A) = \frac{1}{\meas({\mathcal{I}_\mathcal{K}(T)})} \meas \left\{ \tau \in \mathcal{I}_\mathcal{K}(T) : \tilde{H}_m(s +i\tau) \in A \right\}, 
\end{equation*}
\begin{equation*}
\mathcal{Q}(A) = \m \left\{\omega \in \Omega : \tilde{H}_m(s, \omega) \in A \right\}
\end{equation*} 
for $A \in \mathcal{B}(\mathcal{H}(\mathcal{R}))$. 

Then we prove the following two propositions to prove Theorem~\ref{main theorem}. 

\begin{proposition} \label{proposition1}
The probability measure $\mathcal{Q}_T$ converges weakly to $\mathcal{Q}$ as $T \to \infty$.
\end{proposition}

\begin{proposition} \label{proposition2}
The support of the probability measure $\mathcal{Q}$ coincides with $\mathcal{H}(\mathcal{R})$. 
\end{proposition}

The proof of this propositions is in the principle of Bagchi \cite{Ba}. 
Actually, the proof of Proposition~\ref{proposition1} is based on Endo \cite{En} and Kowalski \cite{Ko}, 
and the proof of Proposition~\ref{proposition2} is based on Endo \cite{En} and Nagoshi and Steuding \cite{NS}.

\section{A LIMIT PROBABILITY MEASURE AND ITS SUPPORT}

\subsection{Proof of Proposition~\ref{proposition1}}

\begin{lemma} \label{lemma smooth}
Let $\varphi : [0, \infty) \to \mathbb{C}$ be smooth and assume that $\varphi$ and all its derivatives decay faster than any polynomial at infinity, and let $\hat{\varphi}(s) = \int_{0}^{\infty} \varphi(x) x^{s-1}\, dx$ be the Mellin transform of $\varphi$ on $\Re(s) > 0$.
\begin{enumerate}
\item[(1)] The Mellin transform $\hat{\varphi}$ extends to a meromorphic function on $\Re(s) > -1$, with at most a simple pole at $s = 0$ with residue $\varphi(0)$.
\item[(2)] For any real numbers $-1 < A < B$, the Mellin transform has rapid decay in the strip $A \le \sigma \le B$, in the sense that for any integer $k \ge 1$, there exists a constant $C = C(k, A, B) \ge 0$ such that 
\[
|\hat{\varphi}(\sigma + it) | \le C (1 + |t|)^{-k}.
\]
for all $A \le \sigma \le B$ and $|t| \ge 1$.
\item[(3)] For any $\sigma > 0$ and any $x \ge 0$, we have the Mellin inversion formula 
\[
\varphi(x) = \frac{1}{2 \pi i}\int_{\sigma - i \infty}^{\sigma + i \infty} \hat{\varphi}(s)x^{-s}\, ds.
\]

\end{enumerate}
\end{lemma}
\begin{proof}
See \cite[Proposition~A.3.1]{Ko}.
\end{proof}

Now we fix a real-valued smooth function $\varphi(x)$ on $[0, \infty)$ with compact support satisfying $\varphi(x) = 1$ and $0 \le \varphi(x) \le 1$. 
We note that we can take $\varphi$ which satisfies the assumptions of Lemma~\ref{lemma smooth}.

\begin{lemma} \label{lemma poly1}
Let $y \ge 2$, $|t| \ge y+3$, and $1/2 \le \sigma_\ast \le 1$ be satisfied $\mathcal{L}(z) \ne 0$ on the set \\
$\{z : \sigma_\ast < \Re(z) \le 1, |\Im(z) -t| \le y+2 \}$. 
Then, for $\sigma_\ast < \sigma \le 1$, 
\[
\tilde{H}_m(\sigma + it) = \sum_{2 \le n \le y} \frac{\Lambda_\mathcal{L}(n)}{n^{\sigma+it}(\log{n})^{m+1}} 
+ O_{\mathcal{L}, m}\left(\frac{\log|t|}{(\sigma' -\sigma_\ast)^2}y^{\sigma' -\sigma}\right), 
\]
where $\sigma' = \min\{\sigma_\ast + \frac{1}{\log{y}}, \frac{\sigma + \sigma_\ast}{2} \}$. 
\end{lemma}

\begin{proof}
Using 
\[
\log{\mathcal{L}}(s) = \sum_{|t -\gamma_\mathcal{L}| \le 1} \log{(s-\rho_\mathcal{L})} + O_{\mathcal{L}}(\log{|t|})
\]
for $-\frac{5}{2} \le \sigma \le \frac{7}{2}$, $s \in G_\mathcal{L}$ (see \cite[Lemma2.1]{PS}), we can prove in the same way as \cite[Lemma~2.1]{GS}.

\end{proof}

\begin{lemma} \label{lemma poly2}
Let $1/2< \sigma_3 < \sigma_4 \le 1$ and let $T$ and $y$ satisfying $T \ge y+3$ be sufficiently larger than $\sigma_3, \sigma_4$. 
We put 
\begin{align*}
l(T; \sigma_3, y) = &\left(\bigcup_{\substack{\rho_\mathcal{L} = \beta_\mathcal{L} +i\gamma_\mathcal{L} \\ \beta_\mathcal{L} > \sigma_3 \\ \gamma_\mathcal{L} \in [T/2, {5T}/2]}} (\gamma_\mathcal{L} -(y +3), \gamma_\mathcal{L} +(y +3)) \right) \\
&\cup[T/2, T/2 + (y+3)] \cup[{5T}/2 -(y+3), {5T}/2].
\end{align*}

Then we have 
\[
\tilde{H}_m(\sigma +it) = \sum_{2 \le n \le y} \frac{\Lambda_\mathcal{L}(n)}{n^{\sigma + it} (\log{n})^{m+1}} 
+ O_{\mathcal{L}, m}(y^{\sigma_3 - \sigma_4} (\log{T})^3)
\]
for $\sigma_4 \le \sigma \le 1, t \in [T/2, 5T/2] \setminus l(T; \sigma_3, y)$ and 

\[
\meas(l(T; \sigma_3, y)) =o_\mathcal{L}\left(\frac{yT}{(\log T)^\alpha}\right)
\]
for $\sigma_3 > \sigma_\mathcal{L}$ for any $\alpha > 0$
\end{lemma}
 
\begin{proof}
We apply Lemma~\ref{lemma poly1} with $\sigma_\ast = \sigma_3$. 
Since $y$ is sufficiently larger than $\sigma_3$, we have  
$\sigma' = \min \left\{\sigma_3 + 1/\log{y}, (\sigma + \sigma_3)/2 \right\} = \sigma_3 + 1/\log{y}$ 
and 
\[
\sigma' -\sigma \le \sigma' -\sigma_4 = \sigma_3 -\sigma_4 + \frac{1}{\log{y}}.
\]
By $\sigma' -\sigma_3 = 1/\log{y}$, it holds that
\[
\frac{1}{\sigma' -\sigma_3} = \log{y} \ll \log{T}.
\]
Therefore we have 
\[
\frac{\log|t|}{(\sigma' -\sigma_3)^2} y^{\sigma'-\sigma} 
\ll y^{\sigma_3 -\sigma_4} (\log{T})^3.
\]
The last estimate follows from the zero density estimate (\ref{zero density}). 
\end{proof}

We will apply Lemma~\ref{lemma poly2} with  
$\sigma_3 = \sigma_0$, $\sigma_4 = \sigma_1$ and we put  
\begin{align*}
Y(T) = (\log{T})^\frac{4}{\sigma_1 -\sigma_0},\  
\mathcal{X}_\mathcal{K}(T) = \mathcal{G}_{\sigma_0, |\mathcal{K}| + Y(T) +4} \cap [T, 2T].
\end{align*}

The zero density estimate (\ref{zero density}) yields  
$\mathcal{X}_\mathcal{K}(T) \subset \mathcal{I}_\mathcal{K}(T)$ and $\meas({\mathcal{X}_\mathcal{K}(T)}) \sim T$ as $T \to \infty$. 
We put 
\[
\tilde{H}_{m, X}(s + i\tau) = \sum_{n=2}^{\infty}\frac{\Lambda_\mathcal{L}(n) \varphi(n/X)}{n^{s+ i\tau}(\log{n})^{m+1}}
= \sum_{2 \le n \le X}\frac{\Lambda_\mathcal{L}(n)}{n^{s+ i\tau}(\log{n})^{m+1}} + R_{m, X}(s +i\tau), 
\]

\[
\tilde{H}_{m, X}(s, \omega) = \sum_{n=2}^{\infty}\frac{\Lambda_\mathcal{L}(n) \omega(n) \varphi(n/X)}{n^s(\log{n})^{m+1}}
\]
for $X \ge 2$.

\begin{lemma} \label{lemma int conv}
For any compact subset $C$ of $\mathcal{R}$, we have 
\[
\lim_{X \to \infty} \limsup_{T \to \infty} \frac{1}{\meas(\mathcal{I}_\mathcal{K}(T))} \int_{\mathcal{X}_\mathcal{K}(T)} \sup_{s \in C}|\tilde{H}_m(s + i\tau) -\tilde{H}_{m, X}(s + i\tau)|\, d\tau = 0.
\]
\end{lemma}

\begin{proof}
Let $T$ be sufficiently large.
When $\tau \in \mathcal{X}_\mathcal{K}(T)$,  $\tilde{H}_m(s + i\tau)$ is holomorphic on $\overline{\mathcal{R}}$.
By Couchy's integral formula, we have 
\[
\tilde{H}_m(s +i\tau) - \tilde{H}_{m, X}(s + i\tau)
=\frac{1}{2\pi i}\int_{\partial{\mathcal{R}}} \frac{\tilde{H}_m(z + i\tau) - \tilde{H}_{m, X}(z + i\tau)}{z - s}\, dz
\]
for any $s \in C$ and any $\tau \in \mathcal{X}_\mathcal{K}(T)$. 
Hence we obtain
\begin{align*}
&\int_{\mathcal{X}_\mathcal{K}(T)} \sup_{s \in C}|\tilde{H}_m(s + i\tau) - \tilde{H}_{m, X}(s + i\tau)|\, d\tau \\
&=\int_{\mathcal{X}_\mathcal{K}(T)} \sup_{s \in C}\left|\frac{1}{2\pi i}\int_{\partial{\mathcal{R}}} \frac{\tilde{H}_m(z +i\tau) -\tilde{H}_{m, X}(z + i\tau)}{z - s}\, dz \right|\, d\tau \\
&\le \frac{1}{2\pi \dist(C, \partial{\mathcal{R}})} \int_{\partial{\mathcal{R}}}|dz| \int_{\mathcal{X}_\mathcal{K}(T)}|\tilde{H}_m(z + i\tau) -\tilde{H}_{m, X}(z + i\tau)|\, d\tau \\
&\le \frac{|\partial{\mathcal{R}}|}{2\pi \dist(C, \partial{\mathcal{R}})} \sup_{\sigma: s \in \partial{\mathcal{R}}} \int_{\mathcal{X}'_\mathcal{K}(T)} |\tilde{H}_m(\sigma + it) -\tilde{H}_{m, X}(\sigma + it)|\, dt, 
\end{align*}
where $\mathcal{X}'_\mathcal{K}(T) =[T/2, {5T}/2] \setminus l(T; \sigma_0, Y(T))$, $\dist(C, \partial{\mathcal{R}})$ is minimal distance between $C$ and $\partial{\mathcal{R}}$, and  $|\partial{\mathcal{R}}|$ is the length of $\partial{\mathcal{R}}$.  

Here we have 
\begin{align*}
\int_{\mathcal{X}'_\mathcal{K}(T)} |\tilde{H}_m(\sigma + it) -\tilde{H}_{m, X}(\sigma + it)|\, dt 
&\le \int_{\mathcal{X}'_\mathcal{K}(T)} \bigg|\tilde{H}_m(\sigma +it) - \sum_{2 \le n \le Y(T)} \frac{\Lambda_\mathcal{L}(n)}{n^{\sigma + it} (\log{n})^{m+1}} \bigg|\, dt \\
&+\int_{\mathcal{X}'_\mathcal{K}(T)} \bigg|\sum_{2 \le n \le Y(T)} \frac{\Lambda_\mathcal{L}(n)}{n^{\sigma + it} (\log{n})^{m+1}}-\tilde{H}_{m, X}(\sigma +it)\bigg|\, dt.
\end{align*}
For the first term, we have 
\[
\int_{\mathcal{X}'_\mathcal{K}(T)} \bigg|\tilde{H}_m(\sigma +it) -\sum_{2 \le n \le Y(T)} \frac{\Lambda_\mathcal{L}(n)}{n^{\sigma +it} (\log{n})^{m+1}}\bigg|\, dt \ll_m \frac{\meas(\mathcal{X}'_\mathcal{K}(T))}{\log{T}}
\]
by Lemma~\ref{lemma poly2}.

For the second term, by the Cauchy--Schwarz inequality, we have
\begin{align*}
&\left(\int_{\mathcal{X}'_\mathcal{K}(T)} \bigg|\sum_{2 \le n \le Y(T)} \frac{\Lambda_\mathcal{L}(n)}{n^{\sigma +it} (\log{n})^{m+1}}-\tilde{H}_{m, X}(\sigma +it)\bigg|\, dt \right)^2 \\
&\le \meas(\mathcal{X}'_\mathcal{K}(T)) \int_{T/2}^{{5T}/2} \bigg|\sum_{2 \le n \le Y(T)} \frac{\Lambda_\mathcal{L}(n)}{n^{\sigma +it} (\log{n})^{m+1}}-\tilde{H}_{m, X}(\sigma +it)\bigg|^2\, dt.
\end{align*} 

Using the mean value theorem of Dirichlet polynomials \cite[Corollary~2]{MV}, we obtain 

\begin{align*}
&\int_{T/2}^{{5T}/2} \bigg|\sum_{2 \le n \le Y(T)} \frac{\Lambda_\mathcal{L}(n)}{n^{\sigma +it} (\log{n})^{m+1}}-\tilde{H}_{m, X}(\sigma +it)\bigg|^2\, dt \\
&\ll(T +O(Y(T))) \sum_{X< n \le Y(T)}\frac{|\Lambda_\mathcal{L}(n)|^2 (1-\varphi(n/X))^2}{n^{2\sigma} (\log{n})^{2m+2}} \\
&\ll T\sum_{X <n}\frac{|\Lambda_\mathcal{L}(n)|^2 }{n^{2\sigma} (\log{n})^{2m+2}}.
\end{align*} 

Since $\meas({\mathcal{I}_\mathcal{K}(T)}) \sim T$ as $T \to \infty$ and $\sigma > \sigma_\mathcal{L} > 1/2$, we have 
\begin{align*}
&\limsup_{T \to \infty} \frac{1}{\meas(\mathcal{I}_\mathcal{K}(T))} \int_{\mathcal{X}_\mathcal{K}(T)} \sup_{s \in C}|\tilde{H}_m(s +i\tau) -\tilde{H}_{m, X}(s +i\tau)|\, d\tau \\
&\ll \left( \sum_{X <n}\frac{|\Lambda_\mathcal{L}(n)|^2 }{n^{2\sigma} (\log{n})^{2m+2}}\right)^{\frac{1}{2}}
\to 0
\end{align*}
as $X \to \infty$.
\end{proof}

\begin{lemma} \label{lemma prob1}
The following statements hold.
\begin{enumerate}  
\item[(i)] For almost all $\omega \in \Omega$, $\sum_{p}\frac{b(p)\omega(p)}{p^s (\log{p})^m}$ is convergent in 
$\mathcal{H}(\mathcal{R})$. 
Furthermore, for almost all $\omega \in \Omega$, $\sum_{p}\sum_{k=1}^{\infty}\frac{b(p^k)\omega(p)^k}{p^{ks} (\log{p^k})^m}$ is convergent in $\mathcal{H}(\mathcal{R})$.

\item[(ii)] 
For almost all $\omega \in \Omega$, there exists a constant $C_1 = C_1(\mathcal{K}, \mathcal{L}, \omega) >0$ such that 
\[
|\tilde{H}_m(s, \omega)| \le C_1(|t| + 2)
\]
for $\Re(s) \ge (\sigma_0 + \sigma_1)/2$.

\item[(iii)] 
There exists a constant $C_2 = C_2(\mathcal{K}, \mathcal{L}) >0$ such that 
\[
|\mathbb{E}^\m[|\tilde{H}_m(s, \omega)|]| \le C_2(|t| +2)
\]
for $\Re(s) \ge (\sigma_0 + \sigma_1)/2$.

\end{enumerate}
\end{lemma}

\begin{proof}
(i) It follows from Kolmogorov's three series theorem (see \cite[Theorem~B.10.1]{Ko}) and the property of Dirichlet series. 
\par
(ii) For $u \ge 2, \omega \in \Omega$, we define 
\begin{align*}
&S_u(\omega) = \sum_{2 \le n \le u} \frac{\Lambda_\mathcal{L}(n) \omega(n)}{n^{\sigma_0} (\log{n})^{m+1}}, \\
&A_u(\omega) = \sum_{p \le u}\frac{b(p)\omega(p)}{p^s (\log{p})^m}.
 \end{align*}
Here $A_u(\omega)$ is convergent as $u \to \infty$ for almost all $\omega \in \Omega$ by (i), 
and 

\[
S_u(\omega) = A_u(\omega) + \sum_{\substack{p^k \le u \\ k \ge 2}} \frac{b(p^k) \omega(p^k)}{k^mp^{k\sigma_0} (\log{p})^{m}}
\]
holds.
Since the second term is bounded, there exists $M = M(\mathcal{K}, \mathcal{L}, \omega) >0$ such that 
$|S_u(\omega)| \le M$ for $u \ge 2$ for almost all $\omega \in \Omega$.
We fix such $\omega \in \Omega$. 
Then, for $\Re(s) \ge (\sigma_0 + \sigma_1)/2$, we have  
\begin{align*}
\tilde{H}_m(s, \omega) &= \int_{2^{-}}^{\infty} \frac{dS_u(\omega)}{u^{s -\sigma_0}} \\
                       &= (s -\sigma_0) \int_{2}^{\infty} \frac{S_u(\omega)}{u^{s -\sigma_0+1}}\, du \\
                       &\ll_{\mathcal{K}, \mathcal{L}, \omega} \frac{|s -\sigma_0|}{\sigma -\sigma_0} \\
                       &\ll |t| +2.
\end{align*}
\par
(iii) By the Cauchy--Schwarz inequality, we notice  
\begin{align*}
\mathbb{E}^\m[|S_u(\omega)|] &\le \left(\mathbb{E}^\m\left[\left|\sum_{2 \le n \le u} \frac{\Lambda_\mathcal{L}(n) \omega(n)}{n^{\sigma_0}(\log{n})^{m+1}}\right|^2\right] \right)^{\frac{1}{2}} \\
                     &\le \left(\sum_{2 \le n \le u} \frac{|\Lambda_\mathcal{L}(n)|^2}{n^{2\sigma_0}(\log{n})^{2m+2}} \right)^{\frac{1}{2}} \\
                     &< M.
\end{align*}
Here $M = M(\mathcal{K}, \mathcal{L})>0$ is a constant. 
Thus, for $\Re(s) \ge (\sigma_0 + \sigma_1)/2$, we have
\begin{align*}
\mathbb{E}^{\m}\left[\left|\tilde{H}_m(s, \omega)\right|\right] &\le |s -\sigma_0| \int_{2}^{\infty} \frac{\mathbb{E}^\m[|S_u(\omega)|]}{u^{\sigma -\sigma_0+1}}\, du\\
                                  &\le M \frac{|s -\sigma_0|}{\sigma -\sigma_0} \\
                                  &\ll |t| +2.
\end{align*}                               
\end{proof}

\begin{lemma} \label{lemma prob2}
For any compact subset $C$ of $\mathcal{R}$, we have 

\[
\lim_{X \to \infty} \mathbb{E}^{\m}[\sup_{s \in C}|\tilde{H}_m(s, \omega) -\tilde{H}_{m, X}(s, \omega)|] =0.
\]

\end{lemma}

\begin{proof}
Let $\omega \in \Omega$ satisfying Lemma~\ref{lemma prob1}(i) and (ii). 
We recall 
\begin{align*}
&\tilde{H}_{m, X}(z, \omega) = \sum_{n=2}^{\infty}  \frac{\Lambda_\mathcal{L}(n) \omega(n) \varphi(n/X)}{n^z (\log{n})^{m+1}}, \\
&\varphi(n/X) = \frac{1}{2 \pi i}\int_{c - i \infty}^{c + i \infty} \hat{\varphi}(u)(n/X)^{-u}\, du. 
\end{align*}
Hence it can be expressed as 
\[
\tilde{H}_{m, X}(z, \omega) = \frac{1}{2 \pi i}\int_{c - i \infty}^{c + i \infty} \tilde{H}_{m}(z+u, \omega)\hat{\varphi}(u)X^u \, du.
\]
for all $z \in \partial{\mathcal{R}}$ and $c > 1$.

We put $\delta = (\sigma_1 -\sigma_0)/4$. 
Then, by Lemma~\ref{lemma smooth} and Lemma~\ref{lemma prob1} (ii), we have 
\[
\tilde{H}_m(z, \omega) -\tilde{H}_{m, X}(z, \omega)
= -\frac{1}{2 \pi i}\int_{-\delta - i \infty}^{-\delta + i \infty} \tilde{H}_{m}(z+u, \omega)\hat{\varphi}(u)X^u\, du.
\]
Using Cauchy's integral formula, we derive 
\begin{align*}
&\sup_{s \in C}|\tilde{H}_m(s, \omega) -\tilde{H}_{m, X}(s, \omega)| \\
&= \sup_{s \in C}\bigg|\frac{1}{2 \pi i}\int_{\partial{\mathcal{R}}}\frac{\tilde{H}_m(z, \omega) -\tilde{H}_{m, X}(z, \omega)}{z-s}\, dz\bigg| \\
&\le \frac{1}{2 \pi \dist(C, \partial{\mathcal{R}})} \int_{\partial{\mathcal{R}}}|\tilde{H}_m(z, \omega) -\tilde{H}_{m, X}(z, \omega)|\, |dz| \\
&\le \frac{X^{-\delta}}{4 \pi^2 \dist(C, \partial{\mathcal{R}})} \int_{\partial{\mathcal{R}}}|dz| \int_{-\infty}^{\infty} |\tilde{H}_{m}(-\delta+z +it, \omega)\hat{\varphi}(-\delta +it)|\, dt \\
&\le \frac{X^{-\delta}|\partial{\mathcal{R}}|}{4 \pi^2 \dist(C, \partial{\mathcal{R}})} \sup_{z \in \partial{\mathcal{R}}} \int_{-\infty}^{\infty} |\tilde{H}_{m}(-\delta+z +it, \omega)\hat{\varphi}(-\delta +it)|\, dt.
\end{align*}

Therefore, by Lemma~\ref{lemma prob1} (iii), we have 
\begin{align*}
&\mathbb{E}^{\m}[\sup_{s \in C}|\tilde{H}_m(s, \omega) -\tilde{H}_{m, X}(s, \omega)|] \\
&\le  \frac{X^{-\delta}|\partial{\mathcal{R}}|}{4 \pi^2 \dist(C, \partial{\mathcal{R}})} \sup_{z \in \partial{\mathcal{R}}} \int_{-\infty}^{\infty} \mathbb{E}^\m[|\tilde{H}_{m}(-\delta+z +it, \omega)|]|\hat{\varphi}(-\delta +it)|\, dt \\
&\ll X^{-\delta} \to 0
\end{align*}
as $X \to \infty$.
\end{proof}

We prepare the next materials to prove Proposition~\ref{proposition1}. 
We define the probability measure on $(\mathcal{I}_\mathcal{K}(T), \mathcal{B}(\mathcal{I}_\mathcal{K}(T)))$ by 
\[
\mathbb{P}_T(A) = \frac{1}{\meas({\mathcal{I}_\mathcal{K}(T)})} \meas (A), 
\]
for $A \in \mathcal{B}(\mathcal{I}_\mathcal{K}(T))$. 
Furthermore, let $\mathcal{P}_0$ be a finite set of prime numbers, and we define the probability measure on 
$(\prod_{p \in \mathcal{P}_0} \gamma_p, \mathcal{B}(\prod_{p \in \mathcal{P}_0} \gamma_p))$ by 
\[
\mathbb{H}^{\mathcal{P}_0}_T(A) = \frac{1}{\meas({\mathcal{I}_\mathcal{K}(T)})} \meas \left\{ \tau \in \mathcal{I}_\mathcal{K}(T) : (p^{i\tau})_{p \in \mathcal{P}_0} \in A \right\}, 
\]
$A \in  \mathcal{B}(\prod_{p \in \mathcal{P}_0} \gamma_p)$. 
Then, the next lemma holds. 

\begin{lemma} [see {\cite[Lemma~2.10]{En}}] \label{lemma weak conv} 
The probability measure $\mathbb{H}^{\mathcal{P}_0}_T$ converges weakly to $\m_{\mathcal{P}_0}$ as $T \to \infty$. 
Here  we denote $\m_{\mathcal{P}_0} = \otimes_{p \in \mathcal{P}_0} \m_p$. 
\end{lemma}

\begin{proof} [Proof of Proposition~\ref{proposition1}]
By Portmanteau's theorem (See \cite[Theorem~13.16]{Kl}), 
we shall show that $|\mathbb{E}^{\mathcal{Q}_T}[F] - \mathbb{E}^{\mathcal{Q}}[F]| \to 0$ as $T \to \infty$ for all 
bounded Lipschitz continuous function $F$ : $\mathcal{H}(\mathcal{R}) \to \mathbb{R}$.
Let $F$ be a bounded real-valued Lipschitz function.
Then there exist nonnegative constants $C_1, C_2$ such that
\[
|F(f)| \le C_1,\ |F(f) - F(g)| \le C_2 d(f, g)
\]
for all $f, g \in \mathcal{H}(\mathcal{R})$.
Now  
\begin{align*}
|\mathbb{E}^{\mathcal{Q}_T}[F] - \mathbb{E}^{\mathcal{Q}}[F]| 
&= |\mathbb{E}^{\mathbb{P}_T}[F(\tilde{H}_m(s + i\tau))] - \mathbb{E}^{\m}[F(\tilde{H}_m(s, \omega))]| \\
&\le |\mathbb{E}^{\mathbb{P}_T}[F(\tilde{H}_m(s + i\tau))] - \mathbb{E}^{\mathbb{P}_T}[F(\tilde{H}_{m, X}(s +i\tau))]| \\
&\ + |\mathbb{E}^{\mathbb{P}_T}[F(\tilde{H}_{m, X}(s + i\tau))] - \mathbb{E}^{\m}[F(\tilde{H}_{m, X}(s, \omega))]| \\
&\ + |\mathbb{E}^{\m}[F(\tilde{H}_{m, X}(s, \omega))] - \mathbb{E}^{\m}[F(\tilde{H}_{m}(s, \omega))]|
\end{align*}
holds for $X \ge 2$. 
We estimate each term.

For the first term, 
\begin{align*}
&|\mathbb{E}^{\mathbb{P}_T}[F(\tilde{H}_m(s + i\tau))] -\mathbb{E}^{\mathbb{P}_T}[F(\tilde{H}_{m, X}(s + i\tau))]| \\
&\le \int_{\mathcal{I}_\mathcal{K}(T)} |F(\tilde{H}_m(s + i\tau)) - F(\tilde{H}_{m, X}(s + i\tau))| \, d\mathbb{P}_T(\tau) \\
&\le 2C_1(F) \frac{\meas{(\mathcal{I}_\mathcal{K}(T) \setminus \mathcal{X}_\mathcal{K}(T))}}{\meas(\mathcal{I}_\mathcal{K}(T))} 
+ C_2(F) \sum_{j=1}^{\infty}\frac{1}{2^j} \int_{\mathcal{X}_\mathcal{K}(T)} \frac{d_j(\tilde{H}_m(s + i\tau), \tilde{H}_{m, X}(s + i\tau))}{1 +d_j(\tilde{H}_m(s + i\tau), \tilde{H}_{m, X}(s + i\tau))}\,  d\mathbb{P}_T(\tau). 
\end{align*}
Using $\mathcal{I}_\mathcal{K}(T) \sim T$, $\meas(\mathcal{X}_\mathcal{K}(T)) \sim T$ as $T \to \infty$, Lebesgue's dominated convergence theorem and Lemma~\ref{lemma int conv}, we have 
\begin{align*}
&\lim_{X \to \infty}\limsup_{T \to \infty} |\mathbb{E}^{\mathbb{P}_T}[F(\tilde{H}_m(s + i\tau))] -\mathbb{E}^{\mathbb{P}_T}[F(\tilde{H}_{m, X}(s + i\tau))]| \\
&\ll  \sum_{j=1}^{\infty}\frac{1}{2^j}\lim_{X \to \infty}\limsup_{T \to \infty}  \int_{\mathcal{X}_\mathcal{K}(T)} d_j(\tilde{H}_m(s + i\tau), \tilde{H}_{m, X}(s + i\tau))\,  d\mathbb{P}_T(\tau) = 0. 
\end{align*}

We consider the second term. 
Let 
\[
\mathcal{P}(\varphi, X) = \left\{p : \text{$p$ prime},\ p | \prod_{\substack{n \in \mathbb{N} \\ \varphi(n/X) \ne 0}} n \right\}.
\]
Definition of $\varphi$ implies that $\mathcal{P}(\varphi, X)$ is finite. 
We define the mapping 
\[
\Psi_{m, X}: \prod_{p \in \mathcal{P}(\varphi, X)} \gamma_p \ni (x_p)_{p \in \mathcal{P}(\varphi, X)} 
\mapsto \sum_{\substack{n=2 \\ \varphi(n/X) \ne 0}}^{\infty} \frac{\Lambda_\mathcal{L}(n) \varphi(n/X)}{n^s (\log{n})^{m+1}} \prod_{p | n}x_p^{-\nu(p ; n)} \in \mathcal{H}(\mathcal{R}).
\]
Then, $\Psi_{m, X}$ is continuous, so we have 
\[
\mathbb{E}^{\mathbb{H}^{\mathcal{P}(\varphi, X)}_T \circ \Psi^{-1}_{m, X}}[F]
= \mathbb{E}^{\mathbb{P}_T}[F(\tilde{H}_{m, X}(s +i\tau))],
\]
\[
\mathbb{E}^{\m_{\mathcal{P}(\varphi, X)} \circ \Psi^{-1}_{m, X}}[F] 
= \mathbb{E}^{\m}[F(\tilde{H}_{m, X}(s, \omega))].
\]

Therefore, by the mapping theorem (for example see \cite[Section2, The Mapping Theorem]{Bi}) and Lemma~\ref{lemma weak conv}, we have 
\[
\mathbb{E}^{\mathbb{P}_T}[F(\tilde{H}_{m, X}(s + i\tau))] = \mathbb{E}^{\mathbb{H}^{\mathcal{P}(\varphi, X)}_T \circ \Psi^{-1}_{m, X}}[F]
\xrightarrow[T \to \infty]{} \mathbb{E}^{\m_{\mathcal{P}(\varphi, X)} \circ \Psi^{-1}_{m, X}}[F] 
= \mathbb{E}^{\m}[F(\tilde{H}_{m, X}(s, \omega))].
\]

For the last term, 
using Lebesgue's dominated convergence theorem and Lemma~\ref{lemma prob2}, we have 
\begin{align*}
&\lim_{X \to \infty} |\mathbb{E}^{\m}[F(\tilde{H}_{m, X}(s, \omega))] -\mathbb{E}^{\m}[F(\tilde{H}_{m}(s, \omega))]| \\
&\le C_2(F) \sum_{j=1}^{\infty}\frac{1}{2^j} \lim_{X \to \infty} \mathbb{E}^\m[\sup_{s \in K_j}|\tilde{H}_m(s, \omega) - \tilde{H}_{m, X}(s, \omega)|] = 0.
\end{align*}

Thus $\mathcal{Q}_T$ converges weakly to $\mathcal{Q}$.  
\end{proof}

\subsection{Proof of Proposition~\ref{proposition2}} 

\begin{proposition} \label{density}
Let $\mathcal{L} \in \mathcal{S}$ satisfying $(S6)$. 
We define 
\[
g_p(s, \omega(p)) = \sum_{k=1}^{\infty} \frac{\Lambda_{\mathcal{L}}(p^k) \omega(p)^k}{(\log{p^k})^{m+1} p^{ks} }
 = \sum_{k=1}^{\infty} \frac{b(p^k) \omega(p)^k}{(\log{p^k})^{m} p^{ks}}
\]
for all prime $p$ and $\omega(p) \in \gamma$. 
Then the set of all convergent series 
\[
\sum_p g_p(s, \omega(p))
\]
with $\omega(p) \in \gamma$ is dense in $\mathcal{H}(\mathcal{R})$
\end{proposition}

\begin{proof}
We can prove this proposition in the same way as \cite[Proposition~1]{NS}.
\end{proof}

\begin{lemma} \label{lemma prob3}
Let $G$ be a simply connected region in $\mathbb{C}$.  
Let $\{X_n\}_{n=1}^{\infty}$ be a sequence of independent $\mathcal{H}(G)$-valued random elements and suppose that $X = \sum_{n=1}^{\infty}X_n$ converges almost surely. 
Then the support of $X$ is the closure of the set of all $f \in \mathcal{H}(G)$ which may be written as a convergent series $f = \sum_{n=1}^{\infty} f_n$ with $f_n \in \supp(X_n)$.

\end{lemma}

\begin{proof}
See \cite[Theorem~1.7.10]{La}.
\end{proof}

\begin{proof}[Proof of Proposition~\ref{proposition2}]
We recall that we define 
\[
g_p(s, \omega(p)) = \sum_{k=1}^{\infty} \frac{b(p^k) \omega(p)^k}{(\log{p^k})^{m} p^{ks}}
\]
in Proposition~\ref{density}.
The probability measure $\mathcal{Q}$ is the distribution of $\tilde{H}_m(s, \omega)$ by definition of $\mathcal{Q}$. 
Hence, by Lemma~\ref{lemma prob3}, we have 
\begin{align*}
\supp(\mathcal{Q}) &= \supp(\tilde{H}_m(s, \omega)) = \supp \left(\sum_p g_p(s, \omega(p))\right) \\
&= \overline{\left\{\sum_p x_p : \text{$x_p \in \supp \left(g_p(s, \omega(p))\right)$ $\sum_p x_p$ is convergence in $\mathcal{H}(\mathcal{R})$.} \right\}}.
\end{align*}
Here, we will show that
\[
\supp \left(g_p(s, \omega(p))\right) \supset \left\{g_p(s, z) : z \in \gamma \right\} =:U_p.
\]

Let $z \in \gamma$ and $\varepsilon > 0$ be arbitrary.
Then, we have 

\begin{align*}
&\m \left\{ d \left(g_p(s, \omega(p)), g_p(s, z)\right) < \varepsilon \right\} \\
&= \frac{1}{2\pi} \int_{0}^{1} I_{\left\{d \left(g_p(s, e^{2\pi ix}), g_p(s, z )\right) 
< \varepsilon \right\}}\, dx, 
\end{align*}
where $I_A$ is the indicator function of the set $A$. 
Now, we have
\begin{align*}
d \left(g_p(s, e^{2\pi ix}), g_p(s, z)\right)
&\le \sum_{j=1}^{\infty}\frac{1}{2^j} \sup_{s \in K_j} \left| \sum_{k=1}^{\infty}\frac{b(p^k)e^{2\pi ikx}}{p^{ks}(\log{p^k})^m} - \sum_{k=1}^{\infty} \frac{b(p^k)z^k}{p^{ks}(\log{p^k})^m}\right| \\
&\ll \sum_{j=1}^{\infty}\frac{1}{2^j} \sum_{k=1}^{\infty} \frac{|e^{2\pi ikx} - z^k|}{p^{k(\sigma_1 - \theta)} (k\log{p})^m}\\
&\le |e^{2\pi ix} - z| \sum_{k=1}^{\infty} \frac{k}{p^{k(\sigma_1 - \theta)} (k\log{p})^m}.
\end{align*}
By $\sigma_1 >\theta$, there is a constant $M >0$ such that  
\[
d \left(g_p(s, e^{2\pi ikx}), g_p(s, z)\right) < M|e^{2\pi ix} - z|.
\]
Therefore, we have 
\begin{align*}
&\m \left\{ d \left(g_p(s, \omega(p)),g_p(s, z)\right) < \varepsilon \right\} \\
&\ge \frac{1}{2\pi} \int_{0}^{1} I_{\left\{ M |e^{2\pi ix} - z|< \varepsilon \right\}}\, dx >0.
\end{align*}
Then we obtain $\supp\left(g_p(s, \omega(p))\right) \supset U_p$. 
By the definition of $\mathcal{Q}$ and Proposition~\ref{density}, we conclude 
\begin{align*}
\mathcal{H}(\mathcal{R}) \supset \supp(\mathcal{Q}) \supset \overline{\left\{\sum_p x_p :\text{ $x_p \in U_p$, $\sum_p x_p$ is convergence in $\mathcal{H}(\mathcal{R})$}. \right\}}
= \mathcal{H}(\mathcal{R}).
\end{align*}

\end{proof}

\section{PROOF OF  THE MAIN THEOREM}
\begin{proof}[Proof of Theorem~\ref{main theorem}]
Let $m$ be a nonnegative integer, $\mathcal{K}$ be a compact subset contained in $\mathcal{D}_\mathcal{L}$ with connected complement. 
Then, we take $\mathcal{R}$ with (\ref{definition of R}).  
Assume that $g(s)$ is a continuous function on $\mathcal{K}$ and holomorphic on the interior of $\mathcal{K}$. 
Fix $\varepsilon >0$.
By Mergelyan's theorem, there exists a polynomial $P(s)$ such that
\[
\sup_{s \in \mathcal{K}} |P(s) - g(s)| < \varepsilon.
\]
Here we define an open set of $\mathcal{H}(\mathcal{R})$ by $\Phi(P):= \left\{f(s) \in \mathcal{H}(\mathcal{R}) : \sup_{s \in \mathcal{K}} |f(s) - P(s)| < \varepsilon \right\}$. 
Applying Portmanteau theorem (See \cite[Theorem~13.16]{Kl}), Proposition~\ref{proposition1}, 
and Proposition~\ref{proposition2}, 
we have  
\begin{align*}
&\liminf_{T \to \infty} \frac{1}{T} \meas \left\{\tau \in [T, 2T] :  \sup_{s \in \mathcal{K}} |\tilde{H}_{m}(s +i\tau) - P(s)| < \varepsilon \right\} \\
&\ge \liminf_{T \to \infty} \frac{1}{\meas(\mathcal{I}_\mathcal{K}(T))} \meas \left\{\tau \in \mathcal{I}_\mathcal{K}(T) : \sup_{s \in \mathcal{K}} |\tilde{H}_{m}(s + i\tau) -P(s)| < \varepsilon \right\} \\
&= \liminf_{T \to \infty} \mathcal{Q}_T(\Phi(P)) \ge \mathcal{Q}(\Phi(P)) >0.
\end{align*}
Now, the inequality
\[
\sup_{s \in \mathcal{K}} |\tilde{H}_{m}(s + i\tau) -g(s)| \le \sup_{s \in \mathcal{K}} |\tilde{H}_{m}(s + i\tau) -P(s)| + \sup_{s \in \mathcal{K}} |P(s) - g(s)|
\]
holds. 
Thus, we obtain the conclusion.
\end{proof}

\begin{acknowledgment}

The author would like to thank deeply Professor Kohji Matsumoto and Dr. Kenta Endo for their helpful comments and many valuable advice.

\end{acknowledgment}

\begin{flushleft}
{\footnotesize
{\sc
Graduate School of Mathematics, Nagoya University, Chikusa-ku, Nagoya 464-8602, Japan.
}\\
{\it E-mail address}: {\tt m21029d@math.nagoya-u.ac.jp}
}
\end{flushleft}


\begin{thebibliography}{99}

\bibitem{Ba} B. Bagchi, \emph{Statistical behaviour and universality properties of the Riemann zeta-function and other allied Dirichlet series}, Thesis, Calcutta, Indian Statistical Institute, 1981.

\bibitem{Bi} P. Billingsley, \emph{Convergence of Probability Measures}, 2nd ed, Wiley, NewYork,1999.

\bibitem{Bo} H. Bohr, Zur Theorie der Riemann'schen Zetafunktion im kritischen Streifen, Acta Math. \textbf{40} (1915),  67–100.

\bibitem{BC} H. Bohr and R. Courant, Neue Anwendungen der Theorie der Diophantischen Approximationen auf die Riemannschen Zetafunktion, J. Reine Angew. Math. \textbf{144} (1914), 249–274. 

\bibitem{Co} J. B. Conway, \emph{Functions of One Complex Variabe I}, 2nd ed, Springer, 1978.

\bibitem{En} K. Endo, Universality theorem for the iterated integrals of the logarithm of the Riemann zeta-function, Lith. Math. J. \textbf{62} (2022), 315-332. 

\bibitem{EI} K. Endo and S. Inoue, On the value distribution of iterated integrals of the logarithm of the Riemann zeta-function I: Denseness, Forum Math. \textbf{33} (1), 167–176, 2021. 

\bibitem{GS} A. Granville and K. Soundararajan, Extreme values of $\zeta(1 + it)$, The Riemann Zeta Function an Related Themes: Papers in Honour of Professor K. Ramachandra, 65--80, Ramanujan Math. Soc. Lect. Notes Ser. \textbf{2}, Ramanujan Math. Soc., Mysore, 2006.

\bibitem{In} S. Inoue, On the logarithm of the Riemann zeta-function and its iterated integrals, preprint, arXiv:1909.03643(2019).

\bibitem{KP} J. Kaczorowski and A. Perelli, On the prime number theorem for the Selberg class. Arch. Math, \textbf{80}. 3 (2003), 255-263. 

\bibitem{Kl} A. Klenke. \emph{Probability Theory: A Comprehensive Course}, 3rd ed, Springer Science \& Business Media, 2020.

\bibitem{Ko} E. Kowalski, \emph{An Introduction to Probabilistic Number Therory}, Cambridge University Press, 2021. 

\bibitem{La} A. Laurin\v{c}ikas, \emph{Limit Theorems for the Riemann Zeta-function}, Kluwer, 1996.

\bibitem{LM} A. Laurin\v{c}ikas and K. Matsumoto, The universality of zeta-functions attached to certain cusp forms, Acta Arith., \textbf{98} (2001), 345-359.

\bibitem{Le} Y. Lee, The universality theorem for Hecke \textit{L}-functions, Math. Z. \textbf{271} (2012), 893-909.

\bibitem{Mi} H. Mishou, The universality theorem for Hecke \textit{L}-functions. Acta Arith. \textbf{110} (1), 45–71 (2003)

\bibitem{MV} H. L. Montgomery and R. C. Vaughan, Hilbert's inequality, J. London Math. Soc. (2) \textbf{8} (1974), 73–82. 

\bibitem{NS} H. Nagoshi and J. Steuding, Universality for \textit{L}-functions in the Selberg class, Lith. Math. J. \textbf{50} (2010), 293–311.

\bibitem{PS} \L. Pańkowski and J. Steuding, Extreme values of \textit{L}-functions from the Selberg class, Int. J. Number Theory \textbf{9} (5) (2013), 1113–1124. 

\bibitem{Se} A. Selberg, Old and new conjectures and results about a class of dirichlet series, Proceedings  of the  Amalfi Conference  on  Analytic  Number  Theory, pp. 367–385,  Univ. Salerno, 1992. 

\bibitem{St} J. Steuding, \emph{Value-distribution of \textit{L}-functions}, Lecture Notes in Math. 1877, Springer, 2007.

\bibitem{Ti} E. C. Titchmarsh, \emph{The Theory of the Riemann Zeta-Function}, Second Edition, Edited and
with a preface by D. R. Heath–Brown, The Clarendon Press, Oxford University Press, New
York, 1986.

\bibitem{Vo} S. M. Voronin, Theorem on the universality of the Riemann zeta-function, Izv. Akad. Nauk SSSR Ser. Mat. \textbf{39} (1975) 475–486 (in Russian); Math. USSR Izv. \textbf{9} (1975), 443–453.
\end{thebibliography}
\end{document}